\documentclass{amsart}
\usepackage{amssymb,mathrsfs,MnSymbol}
\usepackage{graphicx}
\usepackage{color}

\def\bN {\mathbf{N}}

\def\bR {\mathbf{R}}

\def\bV {\mathbf{V}}

\def\fH {\mathfrak{H}}

\def\cB {\mathcal{B}}
\def\cC {\mathcal{C}}
\def\cD {\mathcal{D}}
\def\cE {\mathcal{E}}
\def\cF {\mathcal{F}}
\def\cG {\mathcal{G}}
\def\cH {\mathcal{H}}

\def\cL {\mathcal{L}}

\def\cS {\mathcal{S}}

\def\cU {\mathcal{U}}

\def\a {{\alpha}}

\def\g {{\gamma}}
\def\Ga {{\Gamma}}
\def\de {{\delta}}
\def\eps {{\epsilon}}
\def\th {{\theta}}

\def\l {{\lambda}}
\def\L {{\Lambda}}
\def\si {{\sigma}}

\def\om {{\omega}}
\def\Om {{\Omega}}

\def\d {{\partial}}
\def\grad {{\nabla}}
\def\Dlt {{\Delta}}

\def\rstr {{\big |}}

\def\la {\langle}
\def\ra {\rangle}

\def\dd {\,\mathrm{d}}

\newcommand{\Tr}{\operatorname{trace}}

\newcommand{\Dist}{\operatorname{dist}}

\newcommand{\Lip}{\operatorname{Lip}}

\newcommand{\MKd}{\operatorname{dist_{MK,2}}}
\newcommand{\Op}{\operatorname{OP}}

\def\hb {{\hbar}}

\newcommand{\ba}{\begin{aligned}}
\newcommand{\ea}{\end{aligned}}

\newcommand{\be}{\begin{equation}}
\newcommand{\ee}{\end{equation}}

\newcommand{\lb}{\label}

\newtheorem{Thm}{Theorem}[section]

\newtheorem{Prop}[Thm]{Proposition}
\newtheorem{Cor}[Thm]{Corollary}
\newtheorem{Lem}[Thm]{Lemma}
\newtheorem{Def}[Thm]{Definition}

\begin{document}

\title[Time Splitting for Quantum Dynamics]{On the Convergence of Time Splitting Methods\\ for Quantum Dynamics\\ in the Semiclassical Regime}

\author[F. Golse]{Fran\c cois Golse}
\address[F.G.]{CMLS, \'Ecole polytechnique, 91128 Palaiseau Cedex, France}
\email{francois.golse@polytechnique.edu}

\author[S. Jin]{Shi Jin}
\address[S.J.]{School of Mathematical Sciences, Institute of Natural Sciences, MOE-LSC,  Shanghai Jiao Tong University, Shanghai 200240, China}
\email{shijin-m@sjtu.edu.cn}

\author[T. Paul]{Thierry Paul}
\address[T.P.]{CMLS, \'Ecole polytechnique \& CNRS, 91128 Palaiseau Cedex, France}
\email{thierry.paul@polytechnique.edu}

\begin{abstract}
By using the pseudo-metric introduced in [F. Golse, T. Paul: Archive for Rational Mech. Anal. \textbf{223} (2017) 57--94], which is an analogue of the Wasserstein distance of exponent $2$ between a quantum density operator and a classical 
(phase-space) density, we prove that the convergence of time splitting algorithms for the von Neumann equation of quantum dynamics is uniform in the Planck constant $\hbar$. We obtain explicit uniform in $\hbar$ error estimates for the first 
order Lie-Trotter, and the second order Strang splitting methods.
\end{abstract}

\keywords{Evolutionary equations, Time-dependent Schr\"odinger equations, Exponential operator splitting methods, Wasserstein distance}

\subjclass{65L05, 65M12, 65J10, 81C05}

\date{\today}

\maketitle


\section{Introduction}


One of the main challenges in quantum dynamics and high frequency waves is that one needs to numerically resolve the small wave length which is computationally prohibitive \cite{Baderetal,ER,JMS,Hochbrucketal}. When  a numerical method
is developed one would like to know its mesh strategy, namely, the dependence of the time step and mesh size on the wave length $\hbar$ (for a misuse of notation in this article we will not distinguish the difference between the reduced Planck 
constant $\hbar$ and the wave length). 

Finite difference schemes for the Schr\"odinger equation typically require both time step and mesh size in the semiclassical regime (i.e. for $\hbar\ll 1$) to be of order $O(\hbar)$ (see \cite{MPP}), or even $o(\hbar)$. On the other hand, the time 
splitting spectral method can improve the time step to be of order $O(1)$ if only the physical observables are of interest \cite{BJM}. An important mathematical object to understand these mesh strategies is the Wigner transform \cite{Wigner}, 
which is a convenient tool to study the semiclassical limit of the Schr\"odinger equation \cite{GMMP, LP}. In fact, the mesh strategy of $\Delta t=O(1)$, for the time step $\Delta t$, of the time-splitting spectral method can only be understood in 
the Wigner framework, and not in terms of the wave function \cite{BJM}.

Since the solution to the Schr\"odinger equation is oscillatory with wave length of order $\hbar$, if one  uses a standard metric, such as the $L^2$ or Sobolev norm, one would end up with an numerical error of order $O((\Delta t/\hbar)^m)$ 
for some integer $m$ which depends on the order of the method. This will not allow one to see an $\hbar$ independent mesh strategy. The argument of an $\hbar$ independent time-step strategy in \cite{BJM} for the time splitting discretization 
to the linear Schr\"odinger equation, which was also useful in establishing a similar mesh strategy for the nonlinear Erhenfest dynamics \cite{FJS}, was made at a formal level without quantifying the numerical error.One would be interested 
in finding a suitable metric which allows one to establish  such a mesh strategy at the {\it rigorous} level.  In the present paper, we use the pseudo-metric  introduced in \cite{FGTPaulARMA}  to establish a uniform (in $\hbar$) error estimate 
of the time splitting methods for the von Neumann equation (which describes the evolution of mixed quantum states, and reduces to the Schr\"odinger equation in the case of pure quantum states \cite{Breueretal}) in the semiclassical regime.
 

\section{A Pseudo-Metric for the Classical Limit}


\begin{Def} A density operator on $\fH:=L^2(\bR^d)$ is an operator $R\in\cL(\fH)$ such that
$$
R=R^*\ge 0\,,\quad\Tr_\fH(R)=1\,.
$$
The set of all density operators on $\fH$ will be denoted by $\cD(\fH)$. 
\end{Def}

In the definition above, the notation $\cL(\fH)$ designates the algebra of bounded linear operators defined on $\fH$. Henceforth, we also denote by $\cL^p(\fH)$ for all $p\ge 1$ the two-sided ideal of $\cL(\fH)$ whose elements are the operators 
$T\in\cL(\fH)$ such that $|T|^p=(T^*T)^{p/2}$ is a trace-class operator on $\fH$. For instance $\cL^1(\fH)$ and $\cL^2(\fH)$ are respectively the sets of trace-class and of Hilbert-Schmidt operators on $\fH$. The notation $\Tr_\fH(T)$ designates
the trace of $T\in\cL^1(\fH)$.

We denote by $\cD^2(\fH)$ the set of density operators on $\fH$ such that
\be\lb{FiniteEnerg}
\Tr_\fH(R^{1/2}(-\hbar^2\Dlt_y+|y|^2)R^{1/2})<\infty\,.
\ee
If $R\in\cD^2(\fH)$, one has
\be\lb{FiniteEnerg2}
\Tr_\fH((-\hbar^2\Dlt_y+|y|^2)^{1/2}R(-\hbar^2\Dlt_y+|y|^2)^{1/2})=\Tr_\fH(R^{1/2}(-\hbar^2\Dlt_y+|y|^2)R^{1/2})<\infty
\ee
as can be seen from the lemma below (applied to $A=\l^2|y|^2-\hbar^2\Dlt_y$ and $T=R$).

\begin{Lem}\lb{L-FiniteEnerg}
Let $T\in\cL(\fH)$ satisfy $T=T^*\ge 0$, and let $A$ be an unbounded operator on $\fH$ such that $A=A^*\ge 0$. Then 
$$
\Tr_\fH(T^{1/2}AT^{1/2})=\Tr_\fH(A^{1/2}TA^{1/2})\in[0,+\infty]\,.
$$
\end{Lem}

\begin{proof}
The definition of $T^{1/2}$ and $A^{1/2}$ can be found in Theorem 3.35 in chapter V, \S 3 of \cite{Kato}, together with the fact that $A^{1/2}$ and $T^{1/2}$ are self-adjoint. 

If $\Tr_\fH(T^{1/2}AT^{1/2})<\infty$, then $A^{1/2}T^{1/2}\in\cL^2(\fH)$ and the equality holds by formula (1.26) in chapter X, \S 1 of \cite{Kato}.

If $\Tr_\fH(T^{1/2}AT^{1/2})=\infty$, then $\Tr_\fH(A^{1/2}TA^{1/2})=+\infty$, for otherwise $T^{1/2}A^{1/2}$ and its adjoint $A^{1/2}T^{1/2}$ would belong to $\cL^2(\fH)$, so that $T^{1/2}AT^{1/2}\in\cL^1(\fH)$, which would be in contradiction 
with the assumption that $\Tr_\fH(T^{1/2}AT^{1/2})=\infty$.
\end{proof}

Let $f\equiv f(x,\xi)$ be a probability density on $\bR^d\times\bR^d$ such that
\be\lb{2ndMom}
\iint_{\bR^d\times\bR^d}(|x|^2+|\xi|^2)f(x,\xi)dxd\xi<\infty\,.
\ee

\begin{Def} Let $f\equiv f(x,\xi)$ be a probability density on $\bR^{2d}$ and let $R\in\cD(\fH)$. A coupling of $f$ and $R$ is a measurable operator-valued function $(x,\xi)\mapsto Q(x,\xi)$ such that, for a.e. $(x,\xi)\in\bR^d\times\bR^d$,
$$
Q(x,\xi)=Q(x,\xi)^*\ge 0\,,\quad\Tr_\fH(Q(x,\xi))=f(x,\xi)\,,\quad\iint_{\bR^d\times\bR^d}Q(x,\xi)dxd\xi=R\,.
$$
\end{Def}

The second condition above implies that $Q(x,\xi)\in\cL^1(\fH)$ for a.e. $(x,\xi)\in\bR^d\times\bR^d$. Since $\cL^1(\fH)$ is separable, the notion of strong and weak measurability are equivalent for $Q$. The set of couplings of $f$ and $R$ is
denoted by $\cC(f,R)$. Notice that the operator-valued function $(x,\xi)\mapsto f(x,\xi)R$ belongs to $\cC(f,R)$.

In \cite{FGTPaulARMA}, one considers the following ``pseudometric'': for each probability density $f$ on $\bR^d\times\bR^d$ and each $R\in\cD^2(\fH)$,
$$
E_{\hbar}(f,R):=\inf_{Q\in\cC(f,R)}\left(\iint_{\bR^d\times\bR^d}\Tr_\fH(Q(x,\xi)^{1/2}c(x,\xi,y,\hbar D_y)Q(x,\xi)^{1/2})dxd\xi\right)^{1/2}\,,
$$
where the quantum transportation cost is the quadratic differential operator in $y$, parametrized by $(x,\xi)\in\bR^d\times\bR^d$:
$$
c(x,\xi,y,\hbar D_y):=|x-y|^2+|\xi-\hbar D_y|^2\,,\quad D_y:=-i\grad_y\,.
$$

Let $R\in\cD(\fH)$. The Wigner transform of $R$ is 
$$
W_{\hbar}(R)(x,\xi):=\tfrac1{(2\pi)^d}\int_{\bR^d}r(x+\tfrac12\hbar y,x-\tfrac12\hbar y)e^{-i\xi\cdot y}dy\,,
$$
where $r\equiv r(x,y)$ is the integral kernel of $R$. It is a well known fact that $W_{\hbar}(R)$ is real-valued (since $R=R^*$). It is also well known that $W_{\hbar}(R)$ is not necessarily nonnegative a.e. on $\bR^d\times\bR^d$. For instance, 
if $r(X,Y)=\psi(X)\overline{\psi(Y)}$ with $\psi$ odd, one has
$$
W_{\hbar}(R)(0,0)=-\tfrac1{(2\pi)^d}\int_{\bR^d}|\psi(\tfrac12\hbar y)|^2dy<0\,.
$$
The Husimi transform of $R$ henceforth denoted $\tilde W_{\hbar}(R)$ is defined in terms of its Wigner transform by the formula
$$
\tilde W_{\hbar}(R):=e^{\hbar\Dlt_{x,\xi}/4}W_{\hbar}(R)\,.
$$

Finally, we recall the definition of a T\"oplitz operator. The family of Schr\"odinger coherent states is
$$
|q,p\ra(x):=(\pi\hbar)^{-d/4}e^{-|x-q|^2/2\hbar}e^{ip\cdot(x-q/2)/\hbar}\,,\qquad x,q,p\in\bR^d\,.
$$ 
Let $\mu$ be a positive Borel measure on $\bR^d\times\bR^d$; the T\"oplitz operator with symbol $\mu$ is
$$
\Op^T_{\hbar}(\mu):=\tfrac1{(2\pi\hbar)^d}\int_{\bR^d\times\bR^d}|q,p\ra\la q,p|\mu(dqdp)\,.
$$
One easily checks that, if $\mu$ is the Lebesgue measure (denoted by $1$), then 
$$
\Op^T_{\hbar}(1)=I\,.
$$
Moreover, one easily checks that, if $\mu$ is a Borel probability measure on $\bR^d\times\bR^d$, then $\Op^T_{\hbar}((2\pi\hbar)^d\mu)\in\cD(\fH)$. In addition, if $\mu$ has finite second order moment as in \eqref{2ndMom}, then one has
$\Op^T_{\hbar}((2\pi\hbar)^d\mu)\in\cD^2(\fH)$.

\smallskip
The pseudo-metric $E_{\hbar}$ satisfies the following fundamental properties. Henceforth, we denote by $\MKd$ the Monge-Kantorovich or Wasserstein distance with exponent $2$ defined on the set of Borel probability measures satisfying
the finite second order moment condition \eqref{2ndMom} (see chapter 7 in \cite{VillaniAMS}), whose definition is recalled below.

\begin{Def}\lb{D-MKd}
For all $\rho$ and $\rho'$, Borel probability measures on $\bR^{2d}$, we set
$$
\MKd(\rho,\rho'):=\inf_{\pi\in\Pi(\rho,\rho')}\left(\int_{\bR^{2d}}(|q-q'|^2+|p-p'|^2)\pi(dqdpdq'dp')\right)^{1/2}\,,
$$
where $\Pi(\rho,\rho')$ designates the set of couplings of $\rho$ and $\rho'$. More precisely, $\Pi(\rho,\rho')$ is the set of Borel probability measures on $\bR^{2d}\times\bR^{2d}$ with first and second marginals $\rho$ and $\rho'$ resp., i.e. 
such that
$$
\ba
\int_{\bR^{2d}\times\bR^{2d}}&(\phi(q,p)+\phi'(q',p'))\pi(dqdpdq'dp')
\\
&=\int_{\bR^{2d}}\phi(q,p)\rho(dqdp)+\int_{\bR^{2d}}\phi'(q',p')\rho'(dq'dp')
\ea
$$
for all $\phi,\phi'\in C_b(\bR^{2d})$.
\end{Def}

\smallskip
Specifically, the proposition below explains how the pseudo-metric $E_{\hbar}$ compares to the Wasserstein distance $\MKd$.

\begin{Prop}\lb{P-EWass} Let $R\in\cD^2(\fH)$ and let $f$ be a probability distribution on $\bR^d\times\bR^d$ with finite second order moment \eqref{2ndMom}.

\noindent
(a) One has
$$
E_{\hbar}(f,R)^2\ge d\hbar\,.
$$
(b) One has
$$
E_{\hbar}(f,R)^2\ge\MKd(f,\tilde W_{\hbar}(R))^2-d\hbar
$$
(c) For each Borel probability measure $\mu$ on $\bR^d\times\bR^d$ with finite second order moment as in \eqref{2ndMom}, one has
$$
E_{\hbar}(f,\Op^T_{\hbar}((2\pi\hbar)^d\mu))^2\le\MKd(f,\mu)^2+d\hbar\,.
$$
\end{Prop}

\smallskip
The pseudo-metric $E_{\hbar}$ can be used to obtain a quantitative formulation of the classical limit of quantum mechanics, as explained in \cite{FGTPaulARMA}. Henceforth, we denote by $V$ a real-valued function satisfying
\be\lb{HypV}
V^-\in L^{d/2}(\bR^d)\,,\quad\text{ and }\quad V\in C^{1,1}(\bR^d)\,.
\ee
(Here, the notation $V^-$ designates the function $x\mapsto V^-(x):=\max(-V(x),0)$.)

Let $\l\ge 0$, and set
$$
H_\l(x,\xi):=\tfrac12\l|\xi|^2+V(x)\,.
$$
(From the physical point of view, the parameter $\l\ge 0$ which appears in the definition of the Hamiltonian $H_\l$ can be thought of as the reciprocal mass of the particle whose dynamics is defined in terms of the Hamiltonian flow associated
to $H_\l$, whose definition is recalled below. In the present paper, the parameter $\l$ is used only as a convenient notation for defining the various time-splitting algorithms considered.)

Since $V$ satisfies the second condition in \eqref{HypV}, we deduce from the Cauchy-Lipschitz theorem that the Hamiltonian $H_\l$ generates a globally defined flow denoted 
$$
(X(t;x,\xi),\Xi(t;x,\xi))
$$ 
on $\bR^d\times\bR^d$. In other words, $t\mapsto(X(t;x,\xi),\Xi(t;x,\xi))$ is the solution to the Cauchy problem
$$
\dot{X}=\l\Xi\,,\quad\dot{\Xi}=-\grad V(X)\,,\qquad (X(0;x,\xi),\Xi(0;x,\xi))=(x,\xi)\,.
$$
Equivalently, for each probability distribution $f^{in}$ on $\bR^d\times\bR^d$ satisfying \eqref{2ndMom}, the function $f^{in}\circ\Phi_{-t}$, where $\Phi_t$ is the map $(x,\xi)\mapsto\Phi_t(x,\xi):=(X(t;x,\xi),\Xi(t;x,\xi))$, is the solution to the Cauchy 
problem for the Liouville equation
\be\lb{CPLiouv}
\d_tf+\{H_\l,f\}=0\,,\qquad f\rstr_{t=0}=f^{in}\,.
\ee
Here, the notation $\{\cdot,\cdot\}$ designates the Poisson bracket defined on $\bR^d\times\bR^d$ by the relations
$$
\{x_j,x_k\}=\{\xi_j,\xi_k\}=0\,,\quad\{\xi_j,x_k\}=\de_{jk}\,.
$$

Likewise consider the quantum Hamiltonian 
$$
\cH_\l:=-\tfrac12\hbar^2\l\Dlt_y+V(y)\,.
$$
The parameter $\l$ that appears in the definition of the operator $\cH_\l$ has the same meaning, and is used similarly as in the classical setting.

The first condition in \eqref{HypV} implies that $\cH_\l$ has a self-adjoint extension (still denoted by $\cH_\l$) on $\fH$ (see Lemma 4.8b in chapter VI, \S 4 of \cite{Kato}). By the Stone theorem, $U(t):=e^{it\cH_\l/\hbar}$ is a unitary group 
on $\fH$, and, for each $R^{in}\in\cD(\fH)$, the density operator $R(t):=U^*(t)R^{in}U(t)$ is the generalized solution to the Cauchy problem for the von Neumann equation
\be\lb{CPSchro}
i\hbar\d_tR=[\cH_\l,R]\,,\qquad R\rstr_{t=0}=R^{in}\,.
\ee 

\begin{Thm}\lb{T-PropaEstim}
Let $R^{in}\in\cD^2(\fH)$ and let $f^{in}$ be a probability density on $\bR^d\times\bR^d$ satisfying \eqref{2ndMom}. Then
$$
E_{\hbar}(f^{in}\circ\Phi_{-t},U(t)^*R^{in}U(t))\le E_{\hbar}(f^{in},R^{in})\exp\left(\tfrac12t(\l+\max(1,\Lip(\grad V)^2))\right)\,.
$$
\end{Thm}

This is a straightforward variant of Theorem 2.7 in \cite{FGTPaulARMA} with an external potential and without interaction potential (i.e. in the special case $N=n=1$). The parameter $\l\ge 0$ appearing in the statement above is the other 
(unessential) difference with the situation discussed in \cite{FGTPaulARMA}. 


\section{Main Result}


The simple time-splitting method for the von Neumann equation is
\be\lb{SSplitt}
\left\{
\ba
{}&R^0_\hb=R^{in}_\hb\,,
\\
&i\hbar\d_tA_\hb=[-\tfrac12\hb^2\Dlt_x,A_\hb]\,,&&\qquad A_\hb\rstr_{t=0}=R^n_\hb\,,\qquad n\in\bN\,,
\\
&i\hbar\d_tB_\hb=[V(x),B_\hb]\,,&&\qquad B_\hb\rstr_{t=0}=A_\hb(\Dlt t)\,,
\\
&R^{n+1}_\hb=B_\hb(\Dlt t)\,.
\ea
\right.
\ee

\begin{Thm}\lb{T-SSplit}
Let $V$ satisfy \eqref{HypV}, and assume that $R^{in}\in\cD^2(\fH)$ is a T\"oplitz operator on $\fH$. Let $t\mapsto R_{\hbar}(t)$ be the solution of the Cauchy problem \eqref{CPSchro}, and let $R_{\hbar}^n$ be the sequence of density operators
constructed by the simple splitting method \eqref{SSplitt}. Let $T>0$, and pick a time step $\Dlt t\in(0,\tfrac12)$. Then, for each $n=0,\ldots,[T/\Dlt t]$, the simple splitting method satisfies the following error estimate, stated in terms of the quadratic 
Monge-Kantorovich or Wasserstein distance between the Husimi functions of the approximate and the exact quantum density operators:
$$
\ba
\MKd(\tilde W_\hb(R^n),\tilde W_\hb(R(n\Dlt t)))
\\
\le C_T\Dlt t+2\sqrt{d\hb}\left(1+\exp\left(\tfrac12T(1+\max(1,\Lip(\grad V)^2))\right)\right)&\,,
\ea
$$
where the constant $C_T$ depends only on $T$, $\grad V(0)$ and $\Lip(\grad V)$, and is defined in formula \eqref{DefCT} below.
\end{Thm}

\smallskip
Instead of the simple splitting method, one can instead consider the Strang splitting method
\be\lb{StrSplitt}
\left\{
\ba
{}&R^0_\hb=R^{in}_\hb\,,
\\
&i\hbar\d_tA_\hb=[-\tfrac12\hb^2\Dlt_x,A_\hb]\,,&&\qquad A_\hb\rstr_{t=0}=R^n_\hb\,,
\\
&i\hbar\d_tB_\hb=[V(x),B_\hb]\,,&&\qquad B_\hb\rstr_{t=0}=A_\hb(\tfrac12\Dlt t)\,,\qquad n\in\bN\,,
\\
&i\hbar\d_tG_\hb=[-\tfrac12\hb^2\Dlt_x,G_\hb]\,,&&\qquad G_\hb\rstr_{t=0}=B_\hb(\Dlt t)\,,
\\
&R^{n+1}_\hb=G_\hb(\tfrac12\Dlt t)\,.
\ea
\right.
\ee

Strang splitting is a second order (in $\Dlt t$) method, so that the convergence rate obtained in the previous theorem can be improved as indicated below.

\begin{Thm}\lb{T-StrangSplit}
Let $V$ satisfy \eqref{HypV} and
$$
\grad^mV\in L^\infty(\bR^d)\,,\quad m=1,2,3\,.
$$
Let $R^{in}\in\cD^2(\fH)$ be a T\"oplitz operator on $\fH$, and let $t\mapsto R_{\hbar}(t)$ be the solution of the Cauchy problem \eqref{CPSchro}. On the other hand, let $R_{\hbar}^n$ be the sequence of density operators constructed by 
the Strang splitting method \eqref{StrSplitt}. Let $T>0$, and pick a time step $\Dlt t\in(0,\tfrac12)$. Then, for each $n=0,\ldots,[T/\Dlt t]$, the Strang splitting method satisfies the following error estimate, stated in terms of the quadratic 
Monge-Kantorovich or Wasserstein distance between the Husimi functions of the approximate and the exact quantum density operators:
$$
\ba
\MKd(\tilde W_\hb(R^n),\tilde W_\hb(R(n\Dlt t)))
\\
\le D_T\Dlt t^2+2\sqrt{d\hb}\left(1+\exp\left(\tfrac12T(1+\max(1,\Lip(\grad V)^2))\right)\right)&\,,
\ea
$$
where the constant $D_T$ depends only on $T$ and $\|\grad^mV\|_{L^\infty}$ for $m=1,2,3$, and is defined in formula \eqref{DefDT} below.
\end{Thm}

\smallskip
Our strategy is the following. First, Theorem \ref{T-PropaEstim} gives the error between the solution of the von Neumann solution \eqref{CPSchro} and that of the classical Liouville equation \eqref{CPLiouv}. Then we obtain an analogous 
error between the time split von Neumann and the time split Liouville. Finally we estimate the time splitting error of the classical Liouville equation, measured in distance $\text{dist}_{MK2}$.  Then a triangle type inequality leads to the 
results in Theorem \ref{T-SSplit} and \ref{T-StrangSplit}. This strategy is best illustrated by Figure 1.

\begin{figure}

\includegraphics[width=10cm]{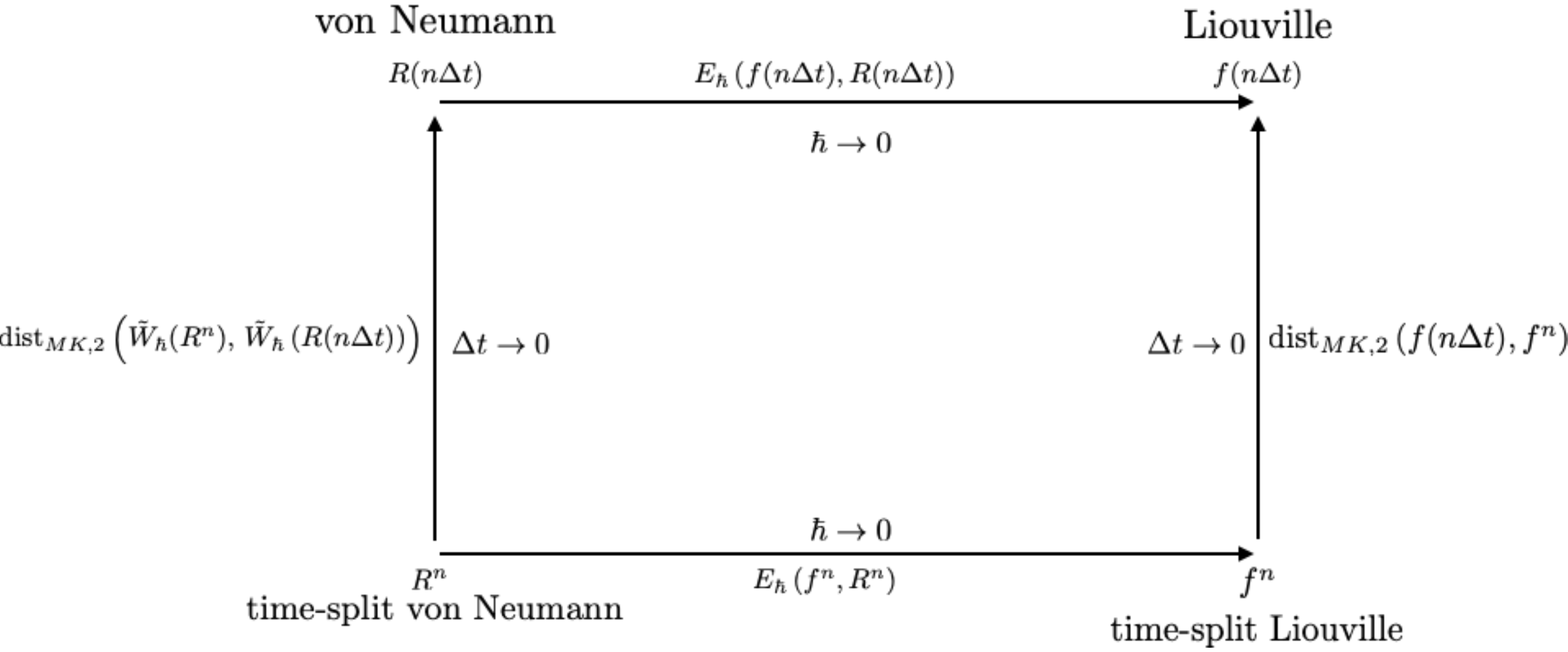}

\caption{Figure 1: the limits $\Delta t\to 0$ and $\hbar\to 0$.}

\end{figure}

\smallskip
The error estimates Theorems \ref{T-SSplit} and \ref{T-StrangSplit} do not provide a uniform in $\hbar$ error estimate, since they contain an $O(\hbar^{1/2})$ term in their right hand side. In particular, these error estimates are useful only 
in the vanishing $\hbar$ limit. Yet these two theorems contain all the new information on the time splitting methods for quantum dynamics in the semiclassical regime that can be obtained with our approach. Besides, these two theorems 
are of independent interest, and lead to better convergence rates that the uniform error estimate given below in the vanishing $\hbar$ regime. In contrast, a classical $L^2$ norm estimate gives an error of order $O(\Delta t/\hbar)^m$ 
\cite{Baderetal} (for some positive integer $m$ that depends on the order of the splitting), which blows up as $\hbar\to 0$.

\smallskip
In order to obtain uniform in $\hbar$ error estimates for the simple and the Strang splitting methods, we need to optimize these estimates with the error estimates for the time splitting method in the case of the Schr\"odinger equation 
with fixed $\hbar$ (or equivalently for $\hbar=1$). Such error estimates have been studied in detail and can be found for instance in \cite{DescombesThalham}. The idea of combining and optimizing the error estimates  in the 
asymptotic (macroscopic) regime and the microscopic regime is often used in numerical methods for kinetic and hyperbolic equations involving multiple scales, a computational methodology known as Asymptotic-Preserving Schemes
 \cite{GJL,Jin-AP}.

\smallskip
Our final uniform error estimate will be formulated in terms of an optimal  transport distance denoted $\Dist_1$, already used in \cite{FGPaulPulvi} (see formula (13) in \cite{FGPaulPulvi}). All the convergence statements in Theorems 
\ref{T-SSplit} and \ref{T-StrangSplit} are ultimately formulated in terms of the Monge-Kantorovich-Wasserstein distance $\MKd$, to which the ``pseudo-metric'' $E_{\hbar}$ can be conveniently compared (see Proposition \ref{P-EWass}), 
the uniform in $\hbar$ error estimates stated below as Corollaries \ref{C-SSplitU} and \ref{C-StrangSplitU} are all based on some optimization procedure comparing the $L^1$ and the $\MKd$ distances between the Husimi transforms 
of the exact and of the approximate solutions of the quantum dynamical problem. This optimization procedure is precisely the reason for using the distance $\Dist_1$, a weaker variant of the Monge-Kantorovich-Wasserstein distance 
of exponent $1$, with a transportation cost that is truncated at infinity. The definition of $\Dist_1$ is recalled below for the reader's convenience.

\begin{Def}
For all $\rho$ and $\rho'$, Borel probability measures on $\bR^{2d}$, we set
$$
\Dist_1(\rho,\rho'):=\inf_{\pi\in\Pi(\rho,\rho')}\int_{\bR^{2d}}\min(1,\sqrt{|q-q'|^2+|p-p'|^2})\pi(dqdpdq'dp')\,.
$$
Here, the notation $\Pi(\rho,\rho')$ designates the set of couplings of $\rho,\rho'$ already used to define the Monge-Kantorovich-Wasserstein distance $\MKd$ (Definition \ref{D-MKd}).
\end{Def}

\smallskip
Our uniform estimates for the simple splitting method is given in the following statement, which is a consequence of Theorem \ref{T-SSplit} and of the error estimate in Theorem 2 of \cite{DescombesThalham}.

\begin{Cor}\lb{C-SSplitU}
Let $V\in C^2(\bR^d)$ satisfy \eqref{HypV}, and let $R^{in}=\Op^T_{\hbar}((2\pi\hbar)^d\mu^{in})$, where $\mu^{in}$ is a Borel probability measure on $\bR^{2d}$ such that
$$
\int_{\bR^{2d}}(|q|^2+|p|^2)\mu^{in}(dqdp)<\infty\,.
$$
Let $t\mapsto R_{\hbar}(t)$ be the solution of the Cauchy problem \eqref{CPSchro}, and let $R_{\hbar}^n$ be the sequence of density operators constructed by the simple splitting method \eqref{StrSplitt}. Let $T>0$, and pick a time step 
$\Dlt t\in(0,\tfrac12)$. Then, for each $n=0,\ldots,[T/\Dlt t]$, the simple splitting method satisfies the following uniform in $\hbar$ error estimate:
$$
\Dist_1(\tilde W_{\hbar}(R_{\hbar}^n),\tilde W_{\hbar}(R_{\hbar}(n\Dlt t)))\le 2C[T,V,\mu^{in}]\Dlt t^{1/3}\,,
$$
where $C[T,V,\mu^{in}]$ is defined in \eqref{DefC[T,V,mu]}. In particular, the constant $C[T,V,\mu^{in}]$ is independent of $\hbar$.
\end{Cor}

\smallskip
Likewise, Theorem \ref{T-StrangSplit} and the error estimate in Theorem 3 of \cite{DescombesThalham} lead to the following statement.

\begin{Cor}\lb{C-StrangSplitU}
Let $V\in W^{4,\infty}(\bR^d)$ satisfy \eqref{HypV}, and let $R^{in}=\Op^T_{\hbar}((2\pi\hbar)^d\mu^{in})$, where $\mu^{in}$ is a Borel probability measure on $\bR^{2d}$ such that
$$
\int_{\bR^{2d}}(|q|^2+|p|^2)\mu^{in}(dqdp)<\infty\,.
$$
Let $t\mapsto R_{\hbar}(t)$ be the solution of the Cauchy problem \eqref{CPSchro}, and let $R_{\hbar}^n$ be the sequence of density operators constructed by the Strang splitting method \eqref{SSplitt}. Let $T>0$, and pick a time step 
$\Dlt t\in(0,\tfrac12)$. Then, for each $n=0,\ldots,[T/\Dlt t]$, the simple splitting method satisfies the following uniform in $\hbar$ error estimate:
$$
\Dist_1(\tilde W_{\hbar}(R_{\hbar}^n),\tilde W_{\hbar}(R_{\hbar}(n\Dlt t)))\le 2D[T,V,\mu^{in}]\Dlt t^{2/3}\,,
$$
where $D[T,V,\mu^{in}]$ is defined in \eqref{DefD[T,V,mu]}. In particular, the constant $D[T,V,\mu^{in}]$ is independent of $\hbar$.
\end{Cor}

\smallskip
As will be clear from the proofs, the uniform in $\hbar$ estimates obtained in these two corollaries involve the $O(\hbar^{1/2})$ term in the convergence rates in Theorems \ref{T-SSplit} and \ref{T-StrangSplit}, and the nonuniform bounds 
in Theorems 2 and 3 resp. of \cite{DescombesThalham}. 

Specifically, the uniform $O(\Dlt t^{1/3})$ error bound in Corollary \ref{C-SSplitU} is obtained as the minimum of the $O(\hbar^{1/2})$ term in Theorem \ref{T-SSplit} and of the (nonuniform) $O(\Dlt t/\hbar)$ error bound in Theorem 2 of 
\cite{DescombesThalham}. This $O(\Dlt t^{1/3})$ uniform error estimate corresponds to the ``worst'' possible distinguished asymptotic regime $\hbar\sim\Dlt t^{2/3}$. Although the $O(\Dlt t)$ term in Theorem \ref{T-SSplit} is smaller than
the $O(\Dlt t^{1/3})$ uniform error estimate in Corollary \ref{C-SSplitU}, the $O(\Dlt t+\hbar^{1/2})$ error estimate in Theorem \ref{T-SSplit} is still of independent interest in all cases where $\hbar$ is small and satisfies $\hbar=O(\Dlt t^2)$.

Likewise, the uniform $O(\Dlt t^{2/3})$ error bound in Corollary \ref{C-StrangSplitU} comes as the minimum of the $O(\hbar^{1/2})$ term in the convergence rates in Theorem \ref{T-StrangSplit} and of the (nonuniform) $O(\Dlt t^2/\hbar^2)$ 
error bound in Theorem 3 of \cite{DescombesThalham}. In this case, the ``worst'' possible distinguished asymptotic regime is $\hbar\sim\Dlt t^{4/5}$. Here again, the $O(\Dlt t^2)$ term in Theorem \ref{T-StrangSplit} is smaller than the 
$O(\Dlt t^{2/3})$ uniform error estimate in Corollary \ref{C-StrangSplitU}. Nevertheless, the $O(\Dlt t^2+\hbar^{1/2})$ error estimate in Theorem \ref{T-StrangSplit} is of interest independently of the uniform $O(\Dlt t^{2/3})$ bound in
Corollary \ref{C-StrangSplitU} whenever $\hbar$ is small and satisfies $\hbar=O(\Dlt t^4)$. Observe that the Strang splitting method is of second order in time in that regime, for the quantum dynamics as well as for the classical dynamics.


\section{The Simple Splitting Algorithm}


\subsection{The Simple Splitting Algorithm for the von Neumann Equation in the Semiclassical Regime}


In this subsection we estimate the error between the time split von Neumann and the time split Liouville equations. By analogy with the simple splitting method \eqref{SSplitt} for the von Neumann equation, consider the simple time-splitting 
method for the Liouville equation:
$$
\left\{
\ba
{}&f^0=f^{in}\,,
\\
&\d_ta+\{\tfrac12|\xi|^2,a\}=0\,,&&\qquad a\rstr_{t=0}=f^n\,,\qquad n\in\bN\,,
\\
&\d_tb+\{V(x),b\}=0\,,&&\qquad b\rstr_{t=0}=a(\Dlt t)\,,
\\
&f^{n+1}=b(\Dlt t)\,.
\ea
\right.
$$

Applying Theorem \ref{T-PropaEstim} to one time step of the free dynamics, i.e. with $V\equiv 0$ and $\l=1$ shows that
$$
E_\hb(a(\Dlt t),A_\hb(\Dlt t))\le E_\hb(f^n,R^n)\exp(\tfrac12\Dlt t)\,.
$$
Next we apply the same Theorem \ref{T-PropaEstim} to the Hamiltonian dynamics defined by the potential $V$, with $\l=0$: thus
$$
\ba
E_\hb(f^{n+1},R^{n+1})=&E_\hb(b(\Dlt t),B_\hb(\Dlt t))
\\
\le&E_\hb(a(\Dlt t),A_\hb(\Dlt t))\exp(\tfrac12\Dlt t\max(1,\Lip(\grad V)^2))\,.
\ea
$$
Putting both estimates together shows that
$$
E_\hb(f^{n+1},R^{n+1})\le E_\hb(f^n,R^n)\exp(\tfrac12\Dlt t(1+\max(1,\Lip(\grad V)^2)))\,.
$$

Let $T>0$; then for each $n=0,\ldots,[T/\Dlt t]+1$, one has
\be\lb{SchroLiouvSplitt}
\ba
E_\hb(f^n,R^n)\le&E_\hb(f^{in},R^{in})\exp(\tfrac12n\Dlt t(1+\max(1,\Lip(\grad V)^2)))
\\
\le&E_\hb(f^{in},R^{in})\exp(\tfrac12T(1+\max(1,\Lip(\grad V)^2)))\,.
\ea
\ee
Observe that the amplification rate $\exp(\tfrac12T(1+\max(1,\Lip(\grad V)^2)))$ in this estimate is uniform in (i.e. independent of) $\hbar$. This is the key point in our analysis.

\subsection{The Simple Splitting Algorithm for the Liouville Equation}


In this subsection, we estimate the distance between the classical Liouville equation and its time split approximation. The error analysis for the simple splitting method is well known in general. However, for our purpose in this paper, we shall 
formulate this error analysis for this splitting method applied to the Liouville equation in terms of the quadratic Monge-Kantorovich or Wasserstein distance.

One expresses the solutions $a$ and $b$ of the kinetic and the potential part of the Liouville evolution as follows, by using the method of characteristics. For the kinetic phase
$$
a(t,y,\eta)=f^n(K_t(y,\eta))\qquad\hbox{ where }K_t(y,\eta):=(y-t\eta,\eta)\,.
$$
As for the potential phase
$$
b(t,y,\eta)=a(\Dlt t,P_t(y,\eta))\qquad\hbox{ where }P_t(y,\eta):=(y,\eta+t\grad V(y))
$$
Hence, one step of simple splitting corresponds to setting
\be\lb{fnfn+1SS}
f^{n+1}=f^n\circ P_{\Dlt t}\circ K_{\Dlt t}\,,
\ee
with 
$$
(y,\eta)\mapsto P_{\Dlt t}\circ K_{\Dlt t}(y,\eta)=(y-\Dlt t\eta,\eta+\Dlt t\grad V(y-\Dlt t\eta))
$$
Since the transformation $P_{\Dlt t}\circ K_{\Dlt t}$ has Jacobian one, the formula \eqref{fnfn+1SS} means that $f^{n+1}(y,\eta)dyd\eta$ is the image of the measure $f^n(y,\eta)dyd\eta$ by $P_{\Dlt t}\circ K_{\Dlt t}$.

\bigskip
Next we seek to estimate the splitting error
$$
\ba
\MKd(f^{in}\circ\Phi_{-(n+1)\Dlt t},f^{in}\circ(P_{\Dlt t}\circ K_{\Dlt t})^{n+1})^2
\\
\le\int|X(-\Dlt t,x,\xi)-(y-\Dlt t\eta)|^2q^n(dxd\xi dyd\eta)
\\
+\int|\Xi(-\Dlt t,x,\xi)-(\eta+\Dlt t\grad V(y-\Dlt t\eta))|^2q^n(dxd\xi dyd\eta)
\ea
$$
where $q^n$ is \textit{any} coupling of $f(n\Dlt t,\cdot,\cdot)$ and $f^n$. 

\bigskip
For this, we seek to bound
$$
|(X,\Xi)(-t,x,\xi)-(Y,H)(-t,y,\eta)|^2\,,
$$
where
$$
(Y,H)(-t,y,\eta)=P_t\circ K_t(y,\eta)=(y-t\eta,\eta+t\grad V(y-t\eta))
$$
is the numerical particle trajectory. First we derive the dynamic equations for $(Y, H)$. Inverting these relations, and denoting $Y_t:=Y(t;y,\eta)$ and $H_t:=H(t;y,\eta)$ for simplicity, we see that
$$
(y,\eta)=(Y_{-t}+t(H_{-t}-t\grad V(Y_{-t})),H_{-t}-t\grad V(Y_{-t}))\,.
$$
Equivalently
$$
(y,\eta)=(Y_t-t(H_t+t\grad V(Y_t)),H_t+t\grad V(Y_t))\,.
$$
Differentiating in time, we find that
$$
\ba
\dot Y=&H+t\dot H+2t\grad V(Y)+t^2\grad^2V(Y)\dot Y\,,
\\
\dot H=&-\grad V(Y)-t\grad^2V(Y)\dot Y\,,
\ea
$$
or equivalently
$$
\ba
\dot Y=&H+t(-\grad V(Y)-t\grad^2V(Y)\dot Y)+2t\grad V(Y)+t^2\grad^2V(Y)\dot Y
\\
=&H+t\grad V(Y)\,,
\\
\dot H=&-\grad V(Y)-t\grad^2V(Y)\cdot H-t^2\grad^2V(Y)\cdot\grad V(Y)\,.
\ea
$$
Thus, we seek to compare the trajectories of the two following differential systems:
$$
\left\{
\ba
\dot X&=\Xi\,,
\\ \dot\Xi&=-\grad V(X)\,,
\ea
\right.
\quad\hbox{ and }\quad
\left\{
\ba\dot Y&=H+t\grad V(Y)\,,
\\ 
\dot H&=-\grad V(Y)-t\grad^2V(Y)\cdot H-t^2\grad^2V(Y)\cdot\grad V(Y)\,.
\ea
\right.
$$

Therefore, we set
$$
\left\{
\ba
(\dot X-\dot Y)&=(\Xi-H)-z(t)\,,
\\ 
(\dot\Xi-\dot H)&=-(\grad V(X)-\grad V(Y))-\zeta(t)\,,
\ea
\right.
$$
with
$$
\left\{
\ba 
z(t)&:=t\grad V(Y)=t\grad V(y+t\eta)\,,
\\ 
\zeta(t)&=-t\grad^2V(y+t\eta)\cdot(\eta-t\grad V(y+t\eta))-t^2\grad^2V(y+t\eta)\cdot\grad V(y+t\eta)\,.
\ea
\right.
$$

Set $E:=|\grad V(0)|$. Then, by the mean value inequality,
$$
\left\{
\ba |z(t)|\le&|t|(E+\|\grad^2V\|_{L^\infty}(|y|+|t||\eta|))\,,
\\ 
|\zeta(t)|\le&\|\grad^2V\|_{L^\infty}|t|(|\eta|+|t|(E+\|\grad^2V\|_{L^\infty}(|y|+|t||\eta|)))
\\
&+t^2\|\grad^2V\|_{L^\infty}(E+\|\grad^2V\|_{L^\infty}(|y|+t\eta))\,.
\ea
\right.
$$
On the other hand
$$
\left\{
\ba
{}&\frac{d}{dt}\tfrac12|X-Y|^2=(X-Y)\cdot(\Xi-H)-z(t)\cdot(X-Y)\,,
\\
&\frac{d}{dt}\tfrac12|\Xi-H|^2=-(\grad V(X)-\grad V(Y))\cdot(\Xi-H)-\zeta(t)\cdot(\Xi-H)\,,
\ea
\right.
$$
so that
$$
\left\{
\ba
{}&\frac{d}{dt}|X-Y|^2\le|X-Y|^2+|\Xi-H|^2+|z(t)|^2+|X-Y|^2\,,
\\
&\frac{d}{dt}|\Xi-H|^2\le\|\grad^2V\|_{L^\infty}(|X-Y|^2+|\Xi-H|^2)+|\zeta(t)|^2+|\Xi-H|^2\,,
\ea
\right.
$$

By Gronwall's inequality, setting $\L:=\max(1,E,\|\grad^2V\|_{L^\infty}))$, one has
$$
\ba
|X-Y|^2(t)+|\Xi-H|^2(t)\le&(|x-y|^2+|\xi-\eta|^2)e^{(2+\L)|t|}
\\
&+\frac{e^{(2+\L)|t|}-1}{2+\L}\sup_{-|t|<s<|t|}(|z(s)|^2+|\zeta(s)|^2)\,.
\ea
$$

Assume that $0\le t\le\tfrac12$ for simplicity; then
$$
\sup_{|s|\le t}(|z(s)|^2+|\zeta(s)|^2)\le\tfrac94\L^2(\tfrac12+\L)^2t^2(1+|y|^2+|\eta|^2)\,.
$$

Choosing at this point an \textit{optimal} coupling $q^n$ of $f(n\Dlt t,\cdot,\cdot)$ and $f^n$ (see Theorem 1.3 in \cite{VillaniAMS} for the existence of an optimal coupling), one has
$$
\ba
\MKd(f((n+1)\Dlt t,\cdot,\cdot),f^{n+1})^2=\int(|X-Y|^2+|\Xi-H|^2)q^{n+1}(dXd\Xi dYdH)
\\
=
\int(|X(-\Dlt t;x,\xi)-Y(-\Dlt t;y,\eta)|^2+|\Xi(-\Dlt t;x,\xi)-H(-\Dlt t;y,\eta)|^2)q^n(dxd\xi dyd\eta)
\\
\le e^{(2+\L)\Dlt t}\int(|x-y|^2+|\xi-\eta|^2)q^n(dxd\xi dyd\eta)
\\
+\tfrac94\L^2(\tfrac12+\L)^2\Dlt t^2\frac{e^{(2+\L)\Dlt t}-1}{2+\L}\left(1+\int(|y|^2+|\eta|^2)f^n(y,\eta)dyd\eta\right)
\\
=\MKd(f(n\Dlt t,\cdot,\cdot),f^n)^2e^{(2+\L)\Dlt t}
\\
+\tfrac94\L^2(\tfrac12+\L)^2\Dlt t^2\frac{e^{(2+\L)\Dlt t}-1}{2+\L}\left(1+\int(|y|^2+|\eta|^2)f^n(y,\eta)dyd\eta\right)
\ea
$$
and it remains to control the last term in the right hand side. 

Since $f^n(y,\eta)dyd\eta$ is the image of the measure $f^{n-1}(y,\eta)dyd\eta$ by the transformation $P_{\Dlt t}\circ K_{\Dlt t}$, one has
$$
\int(|y|^2+|\eta|^2)f^n(y,\eta)dyd\eta=\int(|y+\Dlt t\eta|^2+|\eta-\Dlt t\grad V(y+\Dlt t\eta)|^2)f^{n-1}(y,\eta)dyd\eta
$$
(by substitution in the left hand side), so that
$$
\ba
|y+\Dlt t\eta|^2+|\eta-\Dlt t\grad V(y+\Dlt t\eta)|^2\le(|y|^2+\Dlt t(|y|^2+|\eta|^2)+\Dlt t^2|\eta|^2)
\\
+(1+\Dlt t)|\eta|^2+2(\Dlt t^2+\Dlt t)(E^2+\L^2(|y|^2+\Dlt t(|y|^2+|\eta|^2)+\Dlt t^2|\eta|^2))
\\
\le(|y|^2+|\eta|^2)(1+\Dlt t)^2(1+2\L^2\Dlt t(1+\Dlt t))+2\Dlt t(1+\Dlt t)E^2&\,.
\ea
$$
Denoting 
$$
\mu_n:=\int(|y|^2+|\eta|^2)f^n(y,\eta)dyd\eta\,,
$$
we easily check that
$$
\mu_n\le(1+\Dlt t+2\L^2\Dlt t(1+\Dlt t)^2)(1+\Dlt t)\mu_{n-1}+2\Dlt t(1+\Dlt t)E\,.
$$
Therefore
$$
\ba
\mu_n\le(1+\Dlt t+2\L^2\Dlt t(1+\Dlt t)^2)^n(1+\Dlt t)^n\mu_0
\\
+2\Dlt t(1+\Dlt t)E\frac{(1+\Dlt t+2\L^2\Dlt t(1+\Dlt t)^2)^n(1+\Dlt t)^n-1}{(1+\Dlt t+2\L^2\Dlt t(1+\Dlt t)^2)(1+\Dlt t)-1}
\\
\le\exp(2n\Dlt t(1+\L^2(1+\Dlt t)^2))\mu_0
\\
+2(1+\Dlt t)E\frac{\exp(2n\Dlt t(1+\L^2(1+\Dlt t)^2))-1}{1+(1+\Dlt t)(1+2\L^2(1+\Dlt t)^2)}&\,.
\ea
$$

Thus, we arrive at the inequality
$$
\ba
\MKd(f((n+1)\Dlt t,\cdot,\cdot),f^{n+1})^2\le\MKd(f(n\Dlt t,\cdot,\cdot),f^n)^2e^{(2+\L)\Dlt t}
\\
+\tfrac94\L^2(\tfrac12+\L)^2\Dlt t^2\frac{e^{(2+\L)\Dlt t}-1}{2+\L}\bigg(1+\exp(2n\Dlt t(1+\L^2(1+\Dlt t)^2))\mu_0
\\
+2(1+\Dlt t)E\frac{\exp(2n\Dlt t(1+\L^2(1+\Dlt t)^2))-1}{1+(1+\Dlt t)(1+2\L^2(1+\Dlt t)^2)}\bigg)&\,.
\ea
$$
Iterating in $n$, we conclude that, for $n=0,1,\ldots,[T/\Dlt t]$
\be\lb{LiouvSplitt}
\MKd(f(n\Dlt t,\cdot,\cdot),f^n)\le C_T\Dlt t\,,
\ee
where
\be\lb{DefCT}
\ba
C_T^2:=\tfrac94\L^2(\tfrac12+\L)^2\frac{e^{(2+\L)T}-1}{2+\L}\bigg(1+\exp(2T(1+\L^2(1+\Dlt t)^2))\mu_0
\\
+2(1+\Dlt t)E\frac{\exp(2T(1+\L^2(1+\Dlt t)^2))-1}{1+(1+\Dlt t)(1+2\L^2(1+\Dlt t)^2)}\bigg)&\,.
\ea
\ee

\subsection{Error Estimate for the Simple Splitting Method}


According to Theorem \ref{T-PropaEstim}, for each $n=0,1,\ldots$, one has
$$
E_\hb(f(n\Dlt t,\cdot,\cdot),R(n\Dlt t))\le E_{\hbar}(f^{in},R^{in})\exp\left(\tfrac12n\Dlt t(1+\max(1,\Lip(\grad V)^2))\right)
$$
and in particular
\be\lb{SchroLiouv}
E_\hb(f(n\Dlt t,\cdot,\cdot),R(n\Dlt t))\le E_{\hbar}(f^{in},R^{in})\exp\left(\tfrac12T(1+\max(1,\Lip(\grad V)^2))\right)
\ee
for $n=0,\ldots,[T/\Dlt t]$. Putting together \eqref{SchroLiouvSplitt}, \eqref{LiouvSplitt} and \eqref{SchroLiouv} shows that
$$
\ba
E_\hb(f^n,R^n)+\MKd(f(n\Dlt t,\cdot,\cdot),f^n)+E_\hb(f(n\Dlt t,\cdot,\cdot),R(n\Dlt t))
\\
\le 2E_{\hbar}(f^{in},R^{in})\exp\left(\tfrac12T(1+\max(1,\Lip(\grad V)^2))\right)+C_T\Dlt t&\,.
\ea
$$
According to Proposition \ref{P-EWass} (b) and using the triangle inequality for $\MKd$, we conclude that
$$
\ba
\MKd(\tilde W_\hb(R^n),\tilde W_\hb(R(n\Dlt t)))
\\
\le 2E_{\hbar}(f^{in},R^{in})\exp\left(\tfrac12T(1+\max(1,\Lip(\grad V)^2))\right)+C_T\Dlt t+2\sqrt{d\hb}&\,.
\ea
$$
In particular, if $R^{in}$ is the T\"oplitz operator with symbol $(2\pi\hbar)^df^{in}$, we conclude from Proposition \ref{P-EWass} (c) that
\be\lb{SimpleSplitErrQ}
\ba
\MKd(\tilde W_\hb(R^n),\tilde W_\hb(R(n\Dlt t)))
\\
\le C_T\Dlt t+2\sqrt{d\hb}\left(1+\exp\left(\tfrac12T(1+\max(1,\Lip(\grad V)^2))\right)\right)&\,.
\ea
\ee


\section{The Strang Splitting Algorithm}


In this subsection we estimate the error between the time split von Neumann and the time split Liouville equations. The Strang time-splitting method for the Liouville equation is
$$
\left\{
\ba
{}&f^0=f^{in}\,,
\\
&\d_ta+\{\tfrac12|\xi|^2,a\}=0\,,&&\qquad a\rstr_{t=0}=f^n\,,
\\
&\d_tb+\{V(x),b\}=0\,,&&\qquad b\rstr_{t=0}=a(\tfrac12\Dlt t)\,,\qquad n\in\bN\,,
\\
&\d_tg+\{\tfrac12|\xi|^2,g\}=0\,,&&\qquad g\rstr_{t=0}=b(n\Dlt t)\,,
\\
&f^{n+1}=g(\tfrac12\Dlt t)\,.
\ea
\right.
$$

Applying Theorem \ref{T-PropaEstim} to one time step of the free dynamics, i.e. with $V\equiv 0$ and $\l=1$ shows that
$$
E_\hb(a(\tfrac12\Dlt t),A_\hb(\tfrac12\Dlt t))\le E_\hb(f^n,R^n)\exp(\tfrac14\Dlt t)\,.
$$
Next we apply the same Theorem \ref{T-PropaEstim} to the Hamiltonian dynamics defined by the potential $V$, with $\l=0$: thus
$$
E_\hb(b(\Dlt t),B_\hb(\Dlt t))\le E_\hb(a(\tfrac12\Dlt t),A_\hb(\tfrac12\Dlt t))\exp(\tfrac12\Dlt t\max(1,\Lip(\grad V)^2))\,.
$$
Finally, we apply again Theorem \ref{T-PropaEstim} to the last time step of the free dynamics, so that
$$
\ba
E_\hb(f^{n+1},R^{n+1}_\hb)=E_\hb(g(\tfrac12\Dlt t),G_\hb(\tfrac12\Dlt t))
\le E_\hb(b(\Dlt t),B_\hb(\Dlt t))\exp(\tfrac14\Dlt t)
\\
\le E_\hb(a(\tfrac12\Dlt t),A_\hb(\frac12\Dlt t))\exp(\tfrac14\Dlt t+\tfrac12\Dlt t\max(1,\Lip(\grad V)^2))
\\
\le E_\hb(f^n,R^n)\exp(\tfrac14\Dlt t+\tfrac12\Dlt t\max(1,\Lip(\grad V)^2)+\tfrac14\Dlt t)
\\
= E_\hb(f^n,R^n)\exp(\tfrac12\Dlt t(1+\max(1,\Lip(\grad V)^2))&\,.
\ea
$$
Hence the uniform in $\hb$ estimate \eqref{SchroLiouvSplitt} also holds for the Strang splitting method.

Next we analyze the Strang splitting method for the Liouville equation in terms of the Monge-Kantorovich or Wasserstein distance. With the same notation as in the previous section, we seek to bound
$$
\MKd(f^{in}\circ\Phi_{-(n+1)\Dlt t},f^{in}\circ(K_{\frac12\Dlt t}\circ P_{\Dlt t}\circ K_{\frac12\Dlt t})^{n+1})^2\,.
$$
In order to do so, we seek to bound
$$
|(X,\Xi)(-t,x,\xi)-(Z,\Om)(-t,z,\om)|^2\,,
$$
where the numerical particle trajectory or bi-characteristic flow of the Liouville equation is
$$
\ba
(Z,\Om)(-t,z,\om)=K_{t/2}\circ P_t\circ K_{t/2}(z,\om)
\\
=(z-t\om-\tfrac12t^2\grad V(z-\tfrac12t\om)),\om+t\grad V(z-\tfrac12t\om))&\,.
\ea
$$
Writing $Z_t:=Z(t,z,\om)$ and $\Om_t:=\Om(t,z,\om)$ for simplicity, we first observe that
$$
Z_t-\tfrac12t\Om_t=z+\tfrac12t\om\,,\quad\Om_t=\om-t\grad V(Z_t-\tfrac12\Om_t)
$$

Hence
$$
\ba
\dot Z_t-\tfrac12\Om_t-\tfrac12t\dot\Om_t=\tfrac12\om
\\
\dot\Om_t+\grad V(Z_t-\tfrac12t\Om_t)+t\grad^2V(Z_t-\tfrac12t\Om_t)\frac{d}{dt}(Z_t-\tfrac12t\Om_t)
\\
=\dot\Om_t+\grad V(Z_t-\tfrac12t\Om_t)+\tfrac12t\grad^2V(Z_t-\tfrac12t\Om_t)\om=0&\,,
\ea
$$
so that
$$
\ba
0=\dot Z_t-\tfrac12(\Om_t+\om)-\tfrac12t\dot\Om_t=\dot Z_t-\Om_t-\tfrac12t\grad V(Z_t-\tfrac12t\Om_t)-\tfrac12t\dot\Om_t
\\
=\dot Z_t-\Om_t+\tfrac14t^2\grad^2V(Z_t-\tfrac12t\Om_t)\om&\,.
\ea
$$

Likewise
$$
\ba
0=\dot\Om_t+\grad V(Z_t-\tfrac12t\Om_t)+\tfrac12t\grad^2V(Z_t-\tfrac12t\Om_t)\om
\\
=\dot\Om_t+\grad V(Z_t)-\tfrac12t\grad^2V(Z_t)\Om_t
\\
+\tfrac18t^2\grad^3V(Z_t-\tfrac12\th t\Om_t):\Om_t^{\otimes 2}+\tfrac12t\grad^2V(Z_t-\tfrac12t\Om_t)\om
\\
=\dot\Om_t+\grad V(Z_t)+\tfrac12t(\grad^2V(Z_t-\tfrac12t\Om_t)-\grad^2V(Z_t))\Om_t
\\
+\tfrac12t\grad^2V(Z_t-\tfrac12t\Om_t)(\om-\Om_t)+\tfrac18t^2\grad^3V(Z_t-\tfrac12\th t\Om_t):\Om_t^{\otimes 2}
\\
=\dot\Om_t+\grad V(Z_t)+\tfrac12t(\grad^2V(Z_t-\tfrac12t\Om_t)-\grad^2V(Z_t))\Om_t
\\
+\tfrac12t^2\grad^2V(Z_t-\tfrac12t\Om_t)\grad V(Z_t-\tfrac12t\Om_t)
\\
+\tfrac18t^2\grad^3V(Z_t-\tfrac12\th t\Om_t):\Om_t^{\otimes 2}
\ea
$$

Summarizing, the numerical bi-characteristic field for the Strang splitting method is
$$
\ba
\dot Z_t=&\Om_t-\tfrac14t^2\grad^2V(Z_t-\tfrac12t\Om_t)\om
\\
=&\Om_t+s(t)
\\
\dot\Om_t=&-\grad V(Z_t)-\tfrac12t(\grad^2V(Z_t-\tfrac12t\Om_t)-\grad^2V(Z_t))\Om_t
\\
&-\tfrac12t^2\grad^2V(Z_t-\tfrac12t\Om_t)\grad V(Z_t-\tfrac12t\Om_t)
\\
&-\tfrac18t^2\grad^3V(Z_t-\tfrac12\th t\Om_t):\Om_t^{\otimes 2}
\\
=&-\grad V(Z_t)+\si(t)
\ea
$$
where
$$
\ba
s(t):=&-\tfrac14t^2\grad^2V(z+\tfrac12t\om)\om
\\
\si(t):=&-\tfrac12t(\grad^2V(z+\tfrac12t\om)-\grad^2V(z+\tfrac12t\om+\tfrac12t\Om_t))\Om_t
\\
&-\tfrac12t^2\grad^2V(z+\tfrac12t\om)\grad V(z+\tfrac12t\om)
\\
&-\tfrac18t^2\grad^3V(Z_t-\tfrac12\th t\Om_t):\Om_t^{\otimes 2}
\ea
$$

Here again, we seek to compare the solution $(Z_t,\Om_t)$ to the Strang splitting differential equation with the solution $(X_t,\Xi_t)$ of the Newton system of motion equations, i.e.
$$
\left\{
\ba
\dot X&=\Xi\,,
\\ 
\dot\Xi&=-\grad V(X)\,,
\ea
\right.
$$
Arguing as in the case of the simple splitting method, we observe that
$$
\left\{
\ba
\dot X-\dot Z&=(\Xi-\Om)-s\,,
\\ 
\dot\Xi-\dot\Om&=-(\grad V(X)-\grad V(Z))-\si(t)\,,
\ea
\right.
$$
so that
$$
\ba
\frac{d}{dt}|X-Z|^2=&2(X-Z)\cdot(\Xi-\Om)-2(X-Z)\cdot s
\\
\le&|X-Z|^2+|\Xi-\Om|^2+|X-Z|^2+|s|^2
\\
\frac{d}{dt}|\Xi-\Om|^2=&-2(\Xi-\Om)\cdot(\grad V(X)-\grad V(Z))-2(\Xi-\Om)\cdot\si
\\
\le&\Lip(\grad V)(|\Xi-\Om|^2+|X-Z|^2)+|\Xi-\Om|^2+|\si|^2
\ea
$$
One easily checks that
$$
\ba
|s(t)|^2+|\si(t)|^2\le& t^4(\tfrac16\|\grad^2V\|^2_{L^\infty}|\om|^2+\tfrac1{16}\|\grad^3V\|^2_{L^\infty}|\Om|^4
\\
&+\tfrac14\|\grad^2V\|^2_{L^\infty}\|\grad V\|^2_{L^\infty}+\tfrac1{64}\|\grad^3V\|^2_{L^\infty}|\Om|^2)
\\
\le& t^4(\tfrac16\|\grad^2V\|^2_{L^\infty}|\om|^2+\tfrac1{2}\|\grad^3V\|^2_{L^\infty}(|\om|^4+t^4\|\grad V\|^4_{L^\infty})
\\
&+\tfrac14\|\grad^2V\|^2_{L^\infty}\|\grad V\|^2_{L^\infty}+\tfrac1{32}\|\grad^3V\|^2_{L^\infty}(|\om|^2+t^2\|\grad V\|^2_{L^\infty}))\,.
\ea
$$
Setting
$$
M:=\max(1,\|\grad V\|^2_{L^\infty},\|\grad^2V\|^2_{L^\infty},\|\grad^3V\|^2_{L^\infty})
$$
we see that
$$
\ba
|s(t)|^2+|\si(t)|^2\le& t^4(\tfrac16M|\om|^2+\tfrac12M(|\om|^4+t^4M^2)+\tfrac14M^2+\tfrac1{32}M(|\om|^2+t^2M))
\\
\le&\tfrac12M^3t^4(1+t^2+t^4+|\om|^2+|\om|^4)
\ea
$$

Choosing an optimal coupling $q^n$ of $f(n\Dlt t,\cdot,\cdot)$ and $f^n$, one has
$$
\ba
\MKd(f((n+1)\Dlt t,\cdot,\cdot),f^{n+1})^2\le\int(|X-Z|^2+|\Xi-\Om|^2)q^{n+1}(dXd\Xi dZd\Om)
\\
=
\!\!\int(|X(-\Dlt t;x,\xi)\!-\!Z(-\Dlt t;z,\om)|^2\!+\!|\Xi(-\Dlt t;x,\xi)\!-\!\Om(-\Dlt t;z,\om)|^2)q^n(dxd\xi dzd\om)
\\
\le e^{(2+\L)\Dlt t}\int(|x-z|^2+|\xi-\om|^2)q^n(dxd\xi dzd\om)
\\
+\frac{e^{(2+\L)\Dlt t}-1}{2+\L}\tfrac12M^3\Dlt t^4\left(1+\Dlt t^2+\Dlt t^4+\int(|\om|^2+|\om|^4)f^n(dzd\om)\right)
\\
\le e^{(2+\L)\Dlt t}\MKd(f(n\Dlt t,\cdot,\cdot),f^n)^2
\\
+\frac{e^{(2+\L)\Dlt t}-1}{2+\L}\tfrac12M^3\Dlt t^4\left(1+\Dlt t^2+\Dlt t^4+\int(|\om|^2+|\om|^4)f^n(dzd\om)\right)
\ea
$$
Arguing as in the case of the simple splitting algorithm, one has
$$
\ba
\nu_n:=\int(|\om|^2+|\om|^4)f^n(dzd\om)=\int(|\Om(-\Dlt t;z,\om)|^2+|\Om(-\Dlt t;z,\om)|^4)f^{n-1}(dzd\om)
\\
\le\int((|\om|+\sqrt{M}\Dlt t)^2+(|\om|+\sqrt{M}\Dlt t)^4)f^{n-1}(dzd\om)
\ea
$$
by substitution in the integral on the left hand side, since $f^n(y,\eta)dyd\eta$ is the image of the measure $f^{n-1}(y,\eta)dyd\eta$ by the transformation $K_{\frac12\Dlt t}\circ P_{\Dlt t}\circ K_{\frac12\Dlt t}$. Since
$$
(|\om|+\sqrt{M}\Dlt t)^2\le|\om|^2+\Dlt t(M+|\om|^2)+M\Dlt t^2\le(1+\Dlt t)(|\om|^2+M\Dlt t)
$$
and
$$
(|\om|+\sqrt{M}\Dlt t)^4\le(1+\Dlt t)^2(|\om|^2+M\Dlt t)^2\le(1+\Dlt t)^3(|\om|^4+M^2\Dlt t)\,,
$$
one has
$$
\nu_n\le(1+\Dlt t)^3(\nu_{n-1}+M(1+M)\Dlt t)\le e^{3\Dlt t}(\nu_{n-1}+M(1+M)\Dlt t)
$$
so that
$$
\ba
\nu_n\le &e^{3n\Dlt t}\nu_0+M(1+M)\Dlt te^{3\Dlt t}\frac{e^{3n\Dlt t}-1}{e^{3\Dlt t}-1}
\\
\le &e^{3n\Dlt t}\nu_0+M(1+M)\Dlt t(e^{3n\Dlt t}-1)\,.
\ea
$$
Hence
$$
\ba
\MKd(f((n+1)\Dlt t,\cdot,\cdot),f^{n+1})^2\le e^{(2+\L)\Dlt t}\MKd(f(n\Dlt t,\cdot,\cdot),f^n)^2
\\
+\frac{e^{(2+\L)\Dlt t}-1}{2+\L}\tfrac12M^3\Dlt t^4\left(1+\Dlt t^2+\Dlt t^4+e^{3n\Dlt t}\nu_0+M(1+M)\Dlt t(e^{3n\Dlt t}-1)\right)
\ea
$$
so that, iterating in $n$, 
$$
\ba
\MKd(f(n\Dlt t,\cdot,\cdot),f^n)^2
\\
\le\frac{e^{(2+\L)n\Dlt t}-1}{2+\L}\tfrac12M^3\Dlt t^4\left(1+\Dlt t^2+\Dlt t^4+e^{3n\Dlt t}\nu_0+M(1+M)\Dlt t(e^{3n\Dlt t}-1)\right)
\\
\le\frac{e^{(2+\L)T}-1}{2+\L}M^3\Dlt t^4\left(1+e^{3T}(\nu_0+M^2)\right)
\ea
$$
for $n=0,1,\ldots,[T/\Dlt t]$ with $0<\Dlt t\le\tfrac12$. In other words
\be\lb{LiouvStrangSplitt}
\MKd(f(n\Dlt t,\cdot,\cdot),f^n)\le D_T\Dlt t^2
\ee
for $n=0,1,\ldots,[T/\Dlt t]$ with $0<\Dlt t\le\tfrac12$, with
\be\lb{DefDT}
D_T^2=\frac{e^{(2+\L)T}-1}{2+\L}M^3\left(1+e^{3T}(\nu_0+M^2)\right)\,.
\ee

Putting together \eqref{SchroLiouvSplitt}, \eqref{LiouvStrangSplitt} and \eqref{SchroLiouv} shows that
$$
\ba
E_\hb(f^n,R^n)+\MKd(f(n\Dlt t,\cdot,\cdot),f^n)+E_\hb(f(n\Dlt t,\cdot,\cdot),R(n\Dlt t))
\\
\le 2E_{\hbar}(f^{in},R^{in})\exp\left(\tfrac12T(1+\max(1,\Lip(\grad V)^2))\right)+D_T\Dlt t^2&\,.
\ea
$$
By Proposition \ref{P-EWass} (b) and the triangle inequality for $\MKd$, 
$$
\ba
\MKd(\tilde W_\hb(R^n),\tilde W_\hb(R(n\Dlt t)))
\\
\le 2E_{\hbar}(f^{in},R^{in})\exp\left(\tfrac12T(1+\max(1,\Lip(\grad V)^2))\right)+D_T\Dlt t^2+2\sqrt{d\hb}
\ea
$$
for $n=0,1,\ldots,[T/\Dlt t]$ with $0<\Dlt t\le\tfrac12$.

In particular, if $R^{in}$ is the T\"oplitz operator with symbol $(2\pi\hbar)^df^{in}$, we conclude from Proposition \ref{P-EWass} (c) and the inequality above that
\be\lb{StrangSplitErrQ}
\ba
\MKd(\tilde W_\hb(R^n),\tilde W_\hb(R(\Dlt t)))
\\
\le D_T\Dlt t^2+2\sqrt{d\hb}\left(1+\exp\left(\tfrac12T(1+\max(1,\Lip(\grad V)^2))\right)\right)&\,.
\ea
\ee


\section{Uniform in $\hbar$ Error Estimates}


\begin{proof}[Proof of Corollary \ref{C-SSplitU}] Throughout this section, we denote
$$
\cU(t):=\exp(-it(-\tfrac12\hbar^2\Dlt+V(x))/\hbar)\,,
$$
and
$$
\cU_K(t):=\exp(\tfrac12it\hbar\Dlt)\,,\qquad\cU_V(t):=\exp(-itV(x)/\hbar)\,.
$$

For the first order time splitting, one has
$$
\ba
R_{\hbar}^n-R_{\hbar}(n\Dlt t)
\\
=(\cU_V(\Dlt t)\cU_K(\Dlt t))^nR_{\hbar}^{in}(\cU_K(\Dlt t)^*\cU_V(\Dlt t)^*)^n-\cU(n\Dlt t)R_{\hbar}^{in}\cU(n\Dlt t)^*
\\
=\int_{\bR^d\times\bR^d}\left((\cU_V(\Dlt t)\cU_K(\Dlt t))^n-\cU(n\Dlt t)\right)|q,p\ra\la q,p|(\cU_K(\Dlt t)^*\cU_V(\Dlt t)^*)^n\mu^{in}(dqdp)
\\
+\int_{\bR^d\times\bR^d}\cU(n\Dlt t)|q,p\ra\la q,p|\left((\cU_K(\Dlt t)^*\cU_V(\Dlt t)^*)^n-\cU(n\Dlt t)^*\right)\mu^{in}(dqdp)\,.
\ea
$$
Hence
$$
\|R_{\hbar}^n-R_{\hbar}(n\Dlt t)\|_1\le 2\int_{\bR^d\times\bR^d}\left\|\left(\cU_V(\Dlt t)\cU_K(\Dlt t))^n-\cU(n\Dlt t)\right)|q,p\ra\right\|_{L^2(\bR^d)}\mu^{in}(dqdp)\,.
$$
At this point, we apply Theorem 2 from \cite{DescombesThalham} for the error of the simple splitting scheme:
$$
\ba
\left\|\left(\cU_V(\Dlt t)\cU_K(\Dlt t))^n-\cU(n\Dlt t)\right)|q,p\ra\right\|_{L^2(\bR^d)}
\\
\le 2\frac{\Dlt t}{\hbar}M(V)\left(M(V)t^2+\hbar\|\,|q,p\ra\|_{H^1(\bR^d)}\right)&\,,
\ea
$$
where
$$
M(V):=\max(2\|\grad V\|_{L^\infty(\bR^d)},\|\grad^2V\|_{L^\infty(\bR^d)})\,.
$$
One has
$$
\hbar^2\left(\|\,|q,p\ra\|^2_{L^2(\bR^d)}+\|\,\grad|q,p\ra\|^2_{L^2(\bR^d)}\right)=\hbar^2+|p|^2+\tfrac{d}2\hbar\le 2\hbar^2+|p|^2+d^2
$$
so that
\be\lb{NonUnifSimpleSplit}
\|R_{\hbar}^n-R_{\hbar}(n\Dlt t)\|_1\le 4\frac{\Dlt t}{\hbar}M(V)\left(M(V)t^2+\sqrt{2}\hbar+d+\int_{\bR^{2d}}|p|\mu^{in}(dqdp)\right)\,.
\ee
Next we apply Lemmas 8.2 and 8.1 in \cite{FGPaulPulvi}: 
\be\lb{Interpol}
\ba
\Dist_1(\tilde W_{\hbar}(R_{\hbar}^n),\tilde W_{\hbar}(R_{\hbar}(n\Dlt t)))
\\
\le\min(\|\tilde W_{\hbar}(R_{\hbar}^n)-\tilde W_{\hbar}(R_{\hbar}(n\Dlt t))\|_{L^1(\bR^{2d})},\MKd(\tilde W_{\hbar}(R_{\hbar}^n),\tilde W_{\hbar}(R_{\hbar}(n\Dlt t))))
\\
\le\min(\|R_{\hbar}^n-R_{\hbar}(n\Dlt t)\|_1,\MKd(\tilde W_{\hbar}(R_{\hbar}^n),\tilde W_{\hbar}(R_{\hbar}(n\Dlt t))))&\,.
\ea
\ee
Using \eqref{NonUnifSimpleSplit}, \eqref{SimpleSplitErrQ} to bound the right hand side of \eqref{Interpol} shows that
$$
\ba
\Dist_1(\tilde W_{\hbar}(R_{\hbar}^n),&\tilde W_{\hbar}(R_{\hbar}(n\Dlt t)))
\\
\le\min\Bigg(&4\frac{\Dlt t}{\hbar}M(V)\left(M(V)t^2+\sqrt{2}\hbar+d+\int_{\bR^{2d}}|p|\mu^{in}(dqdp)\right),
\\
&C_T\Dlt t+2\sqrt{d\hbar}\left(1+\exp\left(\tfrac12t(1+\max(1,\Lip(\grad V)^2))\right)\right)\Bigg)
\ea
$$
so that
$$
\ba
\Dist_1(\tilde W_{\hbar}(R_{\hbar}^n),\tilde W_{\hbar}(R_{\hbar}(n\Dlt t)))\le C[T,V,\mu^{in}]\left(\Dlt t+\min\left(\frac{\Dlt t}{\hbar},\sqrt{\hbar}\right)\right)
\\
=C[T,V,\mu^{in}]\left(\Dlt t+\Dlt t^{1/3}\right)
\ea
$$
with
\be\lb{DefC[T,V,mu]}
\ba
C[T,V,\mu^{in}]:=\max\Bigg(4\sqrt{2}M(V),C_T,4M(V)\left(M(V)T^2+d+\int_{\bR^{2d}}|p|\mu^{in}(dqdp)\right)
\\
2\sqrt{d}\left(1+\exp\left(\tfrac12T(1+\max(1,\Lip(\grad V)^2))\right)\right)\Bigg)&\,.
\ea
\ee
\end{proof}

\begin{proof}[Proof of Corollary \ref{C-StrangSplitU}] 
Arguing as in the proof of Corollary \ref{C-SSplitU} for the Strang splitting, we write
$$
\cS(\Dlt t):=\cU_K(\tfrac{\Dlt t}2)\cU_V(\Dlt t)\cU_K(\tfrac{\Dlt t}2)
$$
and
$$
\ba
R_{\hbar}^n-R_{\hbar}(n\Dlt t)=&\cS(\Dlt t)^nR_{\hbar}^{in}(\cS(\Dlt t)^*)^n-\cU(n\Dlt t)R_{\hbar}^{in}\cU(n\Dlt t)^*
\\
=&\int_{\bR^d\times\bR^d}(\cS(\Dlt t)^n-\cU(n\Dlt t))|q,p\ra\la q,p|(\cS(\Dlt t)^*)^n\mu^{in}(dqdp)
\\
&+\int_{\bR^d\times\bR^d}\cU(n\Dlt t)|q,p\ra\la q,p|(\cS(\Dlt t)^*)^n-\cU(n\Dlt t)^*)\mu^{in}(dqdp)\,,
\ea
$$
so that
$$
\|R_{\hbar}^n-R_{\hbar}(n\Dlt t)\|_1\le 2\int_{\bR^d\times\bR^d}\left\|(\cS(\Dlt t)^n-\cU(n\Dlt t))|q,p\ra\right\|_{L^2(\bR^d)}\mu^{in}(dqdp)\,.
$$
By Theorem 3 from \cite{DescombesThalham}:
$$
\left\|(\cS(\Dlt t)^n-\cU(n\Dlt t))|q,p\ra\right\|_{L^2(\bR^d)}\le M'[T,V,\mu^{in}]\frac{\Dlt t^2}{\hbar}
$$
where the constant $M'$ depends on the final time $T$, on $\|V\|_{W^{4,\infty}(\bR^d)}$, and on
$$
\int_{\bR^{2d}}|p|^2\mu^{in}(dqdp)<\infty\,,
$$
since
$$
\hbar\| |q,p\ra\|_{H^1(\bR^d)}=O(|p|)\quad\text{ while }\quad\hbar^2\| |q,p\ra\|_{H^2(\bR^d)}=O(|p|^2)\,.
$$

Thus
$$
\ba
\Dist_1(\tilde W_{\hbar}(R_{\hbar}^n),\tilde W_{\hbar}(R_{\hbar}(n\Dlt t)))
\\
\le\min\left(M'[T,V,\mu^{in}]\frac{\Dlt t^2}{\hbar},D_T\Dlt t+2\sqrt{d\hbar}\left(1+\exp\left(\tfrac12t(1+\max(1,\Lip(\grad V)^2))\right)\right)\right)
\\
\le D[T,V,\mu^{in}]\left(\Dlt t+\min\left(\frac{\Dlt t^2}{\hbar},\sqrt{\hbar}\right)\right)\,,
\ea
$$
where
\be\lb{DefD[T,V,mu]}
D[T,V,\mu^{in}]:=\max\left(D_T,M'[T,V,\mu^{in}],2\sqrt{d}\left(1+\exp\left(\tfrac12T(1+\max(1,\Lip(\grad V)^2))\right)\right)\right)\,.
\ee
Optimizing in $\hbar$ leads to
$$
\Dist_1(\tilde W_{\hbar}(R_{\hbar}^n),\tilde W_{\hbar}(R_{\hbar}(n\Dlt t)))\le D[T,V,\mu^{in}](\Dlt t+\Dlt t^{2/3})
$$
corresponding to the choice $\hbar=\Dlt t^{4/3}$.
\end{proof}

\bigskip
\noindent
\textbf{Acknowledgements.} The work of Fran\c cois Golse and Thierry Paul was partly supported by LIA LYSM (co-funded by AMU, CNRS, ECM and INdAM). The work of Shi Jin was supported by NSFC grants Nos. 31571071 and 11871297.


\end{document}